\documentclass[12,reqno]{amsart}
\usepackage{amsfonts,mathrsfs,amscd,amssymb,amsthm,amsmath,bm,graphicx,psfrag,subfigure,xypic}

\newtheorem{theorem}{Theorem}[section]
\newtheorem{thm}[theorem]{Theorem}

\newtheorem{lemma}[theorem]{Lemma}

\newtheorem{cor}[theorem]{Corollary}

\numberwithin{equation}{section}

\def\sp{\hspace{2ex}}
\def\supp{{\rm{supp\,}}}

\begin{document}

\title[Indecomposable Integral Flows on Signed Graphs]
{Classification of Indecomposable Integral Flows on Signed Graphs}

\author{Beifang Chen}
\address{Department of Mathematics, Hong Kong University of Science and Technology,
Clear Water Bay, Kowloon, Hong Kong}
\email{mabfchen@ust.hk}
\thanks{Research is supported by RGC Competitive Earmarked Research
Grants 600608 and 600409}

\author{Jue Wang}
\address{Department of Mathematics and Physics, Shenzhen Polytechnic,
Shenzhen, Guangdong Province, 518088, P.R. China}
\email{twojade@alumni.ust.hk}




\maketitle

\begin{abstract}
An indecomposable flow $f$ on a signed graph $\Sigma$ is a
nontrivial integral flow that cannot be decomposed into $f=f_1+f_2$,
where $f_1,f_2$ are nontrivial integral flows having the same sign
(both $\geq 0$ or both $\leq 0$) at each edge. This paper is to
classify indecomposable flows into characteristic vectors of
Eulerian cycle-trees --- a class of signed graphs having a kind of
tree structure in which all cycles can be viewed as vertices of a
tree. Moreover, each indecomposable flow other than circuit flows
can be further decomposed into a sum of certain half circuit flows
having the same sign at each edge. The variety of indecomposable
flows of signed graphs is much richer than that of ordinary unsigned
graphs.
\end{abstract}



\section{Introduction}

A {signed graph} is an ordinary graph whose each edge is endowed
with either a positive sign or a negative sign. The system was
formally introduced by Harary~\cite{Harary} who characterized
balanced signed graphs up to switching, and was much developed by
Zaslavsky~\cite{Zaslavsky1, Zaslavsky3} who successfully extended
most important notions of ordinary graphs to signed graphs, such as
circuit, bond, orientation, incidence matrix, Laplacian, and
associated matroids, etc. Based on Zaslavsky's work, Chen and
Wang~\cite{Chen and Wang 1} introduced flow and tension lattices of
signed graphs and obtained fundamental properties on flows and
tensions, including a few characterizations of cuts and bonds.

Now it is natural to ask, inside the flow and tension lattices, how
integral flows and tensions are build up from more basic integral
flows and tensions. More specifically, what does an integral flow or
tension look like if it {\em cannot} be decomposed at integer scale
but can be possibly decomposed further at fractional scale? The
answer is not only interesting but fundamental in nature because if
one considers circuit flows to be at {atomic level} then
indecomposable flows are at {molecular level}.

For ordinary graphs it is easy to see that indecomposable flows are
simply the graph circuit flows at integer scale. For signed graphs,
however, we shall see that indecomposable flows are much richer than
that of unsigned graphs because in addition to circuit flows, the
fixed spin (signs on edges) produces a new class of characteristic
vectors of so-called {\em directed Eulerian cycle-trees}, which are
not decomposable at integer scale while decomposable (into signed
graph circuit flows) at half-integer scale. The present paper is to
present the solution of such a new phenomenon on signed graphs.

Let $\Sigma=(V,E,\sigma)$ be a signed graph throughout, where
$(V,E)$ is an ordinary finite graph with possible loops and multiple
edges, $V$ is the vertex set, $E$ is the edge set, and
$\sigma:E\rightarrow\{-1,1\}$ is the sign function. Each edge subset
$S\subseteq E$ induces a {\em signed subgraph}
$\Sigma(S):=(V,S,\sigma_S)$, where $\sigma_S$ is the restriction of
$\sigma$ to $S$. A {\em cycle} of $\Sigma$ is a simple closed path.
The \emph{sign} of a cycle is the product of signs on its edges. A
cycle is said to be {\em balanced (unbalanced)} if its sign is
positive (negative). A signed graph is said to be {\em balanced} if
all cycles are balanced, and {\em unbalanced} if one of its cycles
is unbalanced. A connected component of $\Sigma$ is called a {\em
balanced (unbalanced) component} if it is balanced (unbalanced) as a
signed subgraph.

An {\em orientation} of a signed graph $\Sigma$ is an assignment
that each edge $e$ is assigned two arrows at its end-vertices $u,v$
as follows: (i) if $e$ is a positive edge, the two arrows are in the
same direction; (ii) if $e$ is a negative edge, the two arrows are
in opposite directions; see Figure~\ref{Orientation}.
\begin{figure}[h]
\includegraphics[width=13.4mm]{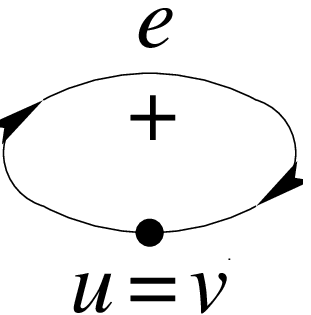}\hspace{11.5mm}
\includegraphics[width=13.4mm]{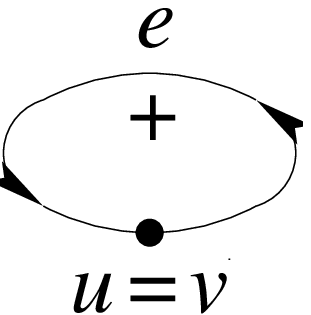}\hspace{11.5mm}
\includegraphics[width=13.4mm]{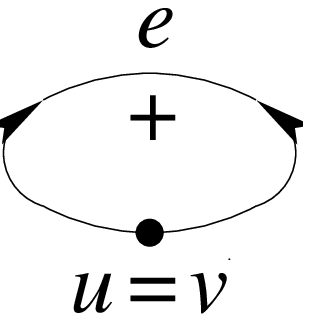}\hspace{11.5mm}
\includegraphics[width=13.4mm]{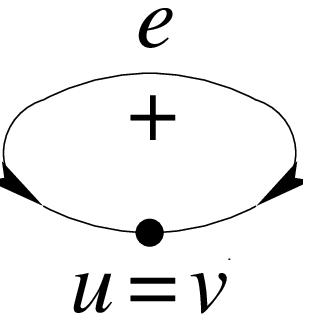}\\ \vspace{3mm}
\includegraphics[width=20mm]{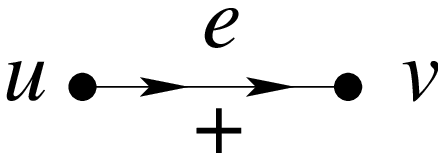}\hspace{5mm}
\includegraphics[width=20mm]{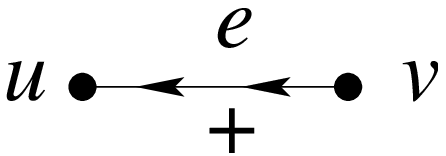}\hspace{5mm}
\includegraphics[width=20mm]{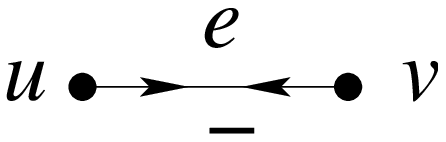}\hspace{5mm}
\includegraphics[width=20mm]{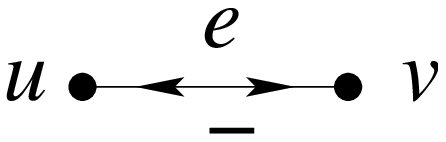}
\caption{Orientations on loops and non-loop edges.}
\label{Orientation}
\end{figure}
We may think of an arrow on an edge $e$ at its one end-vertex $v$ as
$+1$ if the arrow points away from $v$ and $-1$ if the arrow points
toward $v$. Then there are both $+1$ and $-1$ for a positive loop at
its unique end-vertex, and two $+1$'s or two $-1$'s for a negative
loop at its unique end-vertex. So an orientation on $\Sigma$ can be
considered as a multi-valued function $\varepsilon: V\times
E\rightarrow\{-1,0,1\}$ such that
\begin{enumerate}
\item[(i)] $\varepsilon(v,e)$ has two values $+1$ and $-1$ if
$e$ is a positive loop at its unique end-vertex $v$, and is
single-valued otherwise;

\item[(ii)] $\varepsilon(v,e)=0$ if $v$ is not an end-vertex of the edge $e$;
and

\item[(iii)] $\varepsilon(u,e)\varepsilon(v,e)=-\sigma(e)$, $e=uv$.
\end{enumerate}
A signed graph $\Sigma$ with an orientation $\varepsilon$ is called
an {\em oriented signed graph} $(\Sigma,\varepsilon)$. We assume
that $(\Sigma,\varepsilon)$ is an oriented signed graph throughout
the whole paper.

Let $\varepsilon_i$ ($i=1,2$) be orientations on signed subgraphs
$\Sigma_i$ of $\Sigma$. The \emph{coupling} of
$\varepsilon_1,\varepsilon_2$ is a function
$[\varepsilon_1,\varepsilon_2]: E \rightarrow{\Bbb Z}$, defined for
$e=uv$ by
\begin{eqnarray}\label{Coupling-Def}
[\varepsilon_1,\varepsilon_2](e)=\left\{
\begin{array}{rl}
 1 & \hbox{if $e\in \Sigma_1\cap\Sigma_2$, $\varepsilon_1(v,e)=\varepsilon_2(v,e)$,} \\
-1 & \hbox{if $e\in \Sigma_1\cap \Sigma_2$, $\varepsilon_1(v,e)\neq \varepsilon_2(v,e)$,} \\
 0 & \hbox{otherwise.}
\end{array}
\right.
\end{eqnarray}
In other words,
$[\varepsilon_1,\varepsilon_2](e)=\varepsilon_1(v,e)\varepsilon_2(v,e)$
if $e=uv$.

Let $A$ be an abelian group and be assumed automatically a ${\Bbb
Z}$-module. For each edge $e$ and its end-vertices $u,v$, let ${\rm
End}(e)$ denote the multiset $\{u,v\}$. Associated with
$(\Sigma,\varepsilon)$ is the {\em boundary operator} $\partial:
{A}^E\rightarrow {A}^V$ defined for $f\in A^E$ and $v\in V$ by
\begin{align}
(\partial f)(v) & = \sum_{e\in E} {\bm m}_{v,e}f(e) =\sum_{e\in
E,\,u\in{\rm End}(e),\,u=v}\varepsilon(u,e)f(e),
\end{align}
where ${\bm m}_{v,e}=\varepsilon(v,e)$ if $e$ is a non-loop, ${\bm
m}_{v,e}=2\varepsilon(v,e)$ if $e$ is a negative loop, and ${\bm
m}_{v,e}=0$ otherwise. A function $f:E\rightarrow A$ is said to be
{\em conservative} at a vertex $v$ if $(\partial f)(v)=0$, and is
said to be an {\em $A$-flow} of $(\Sigma,\varepsilon)$ if it is
conservative at each vertex of $\Sigma$. The {\em support} of a flow
$f$ is the edge subset
\begin{equation}
\supp f=\{e\in E\mid f(e)\neq 0\}.
\end{equation}
The set of all $A$-flows forms an abelian group, called the {\em
flow group} of $(\Sigma,\varepsilon)$ with values in $A$, denoted
$F(\Sigma,\varepsilon;A)$. We call
$F(\Sigma,\varepsilon):=F(\Sigma,\varepsilon;{\Bbb R})$ the {\em
flow space}, and $Z(\Sigma,\varepsilon): =F(\Sigma,\varepsilon;{\Bbb
Z})$ the {\em flow lattice} of $(\Sigma,\varepsilon)$. For further
information about flows of signed graphs, see
\cite{Beck-Zaslavsky1,Bouchet,Chen and Wang 1, Chen and Wang 2,
Khelladi}. For notions of ordinary graphs, we refer to the books
\cite{Bollobas,Bondy-Murty1,Godsil-Royle}.

A flow is said to be {\em nontrivial} if its support is nonempty. A
nontrivial integral flow $f$ is said to be {\em decomposable} if $f$
can be written as
\[
f=f_1+f_2,
\]
where $f_i$ are nontrivial integral flows having the same sign (both
nonnegative or both nonpositive) at every edge, that is,
$f_1(e)f_2(e)\geq 0$ for all $e\in E$. Nontrivial integral flows
that are not decomposable are called {\em indecomposable flows}. An
integral flow $f$ is said to be {\em elementary} if it is
indecomposable and there is no nontrivial integral flow $g$ such
that $\supp g$ is properly contained in $\supp f$. Compare with
Tutte's definition of elementary chains \cite{Tutte2,Tutte3}.

Let $W$ be a walk of length $n$ in $\Sigma$ and be written as a
vertex-edge sequence
\begin{equation}\label{Vertex-Edge-Seq}
W=u_0x_1u_1x_2\ldots u_{n-1}x_{n}u_{n},
\end{equation}
where each $x_i$ is an edge with end-vertices $u_{i-1},u_{i}$. The
walk $W$ is said to be {\em closed} if the {\em initial vertex}
$u_0$ is the same as the {\em terminal vertex} $u_{n}$. The {\em
sign} of $W$ is the product
\begin{equation}
\sigma(W)=\prod_{i=1}^n\sigma(x_i).
\end{equation}
The {\em support} of $W$ is the set $\supp W$ of edges $x_i$
($i=1,\ldots,n$) with no repetition. We may think of $W$ as the
multiset $\{x_1,x_2,\ldots,x_n\}$ (with repetition allowed) of $n$
edges on $\supp W$.

A \emph{direction} of $W$ is a function $\varepsilon_W$ with values
either $1$ or $-1$, defined for all vertex-edge pairs
$(u_{i-1},x_i)$ and $(u_i,x_i)$, such that
\begin{equation}
\varepsilon_W(u_{i-1},x_i)\varepsilon_W(u_i,x_i)=-\sigma(x_i), \sp
\varepsilon_W(u_i,x_{i})+\varepsilon_W(u_i,x_{i+1})=0.
\end{equation}
Every walk has exactly two opposite directions. A walk $W$ with a
direction $\varepsilon_W$ is called a {\em directed walk}, denoted
$(W,\varepsilon_W)$, and is further called a {\em directed closed
walk} if $u_0=u_{n}$ and
\[
\varepsilon_W(u_0,x_1)+\varepsilon_W(u_{n},x_n)=0.
\]

A directed walk $(W,\varepsilon_W)$ is said to be {\em midway-back
avoided}, provided that if $u_\alpha=u_\beta$ with
$0\leq\alpha<\beta<n$ in (\ref{Vertex-Edge-Seq}) then
\begin{equation}
\varepsilon_W(u_\beta,x_\beta)=\varepsilon_W(u_\alpha,x_{\alpha+1}).
\end{equation}
Figure~\ref{Double-vertex-pattern} demonstrates four possible
orientation patterns at a double vertex in a directed midway-back
avoided walk.
\begin{figure}[h]
\includegraphics[width=25mm]{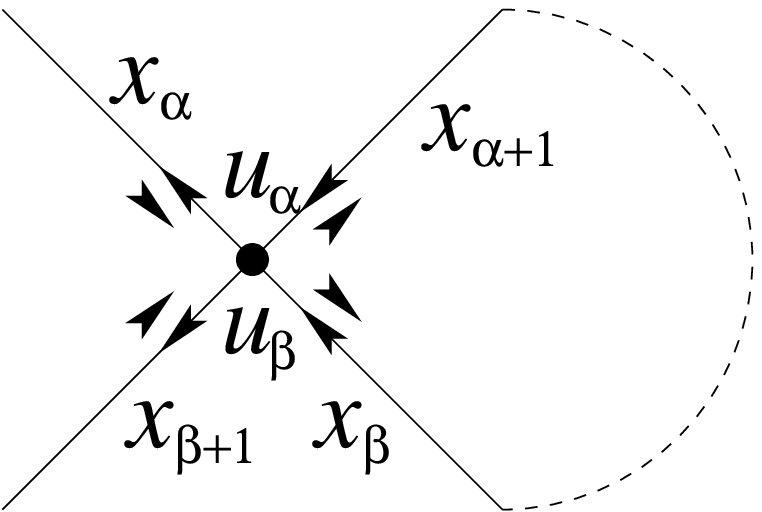}\hspace{2mm}
\includegraphics[width=25mm]{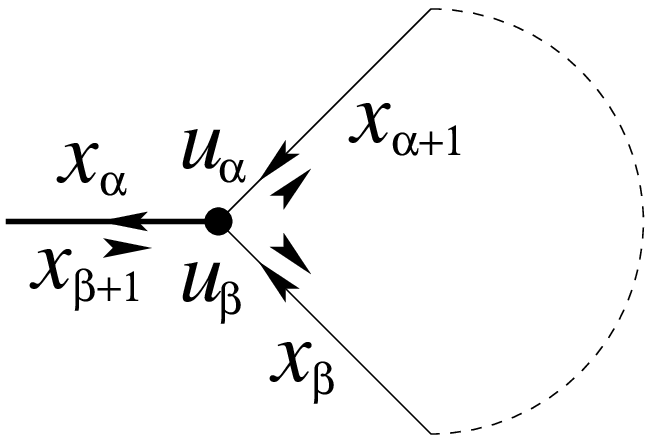}\hspace{2mm}
\includegraphics[width=25mm]{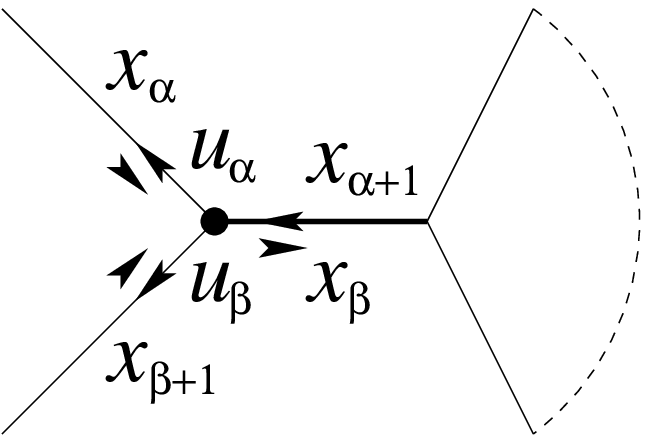}\hspace{2mm}
\includegraphics[width=25mm]{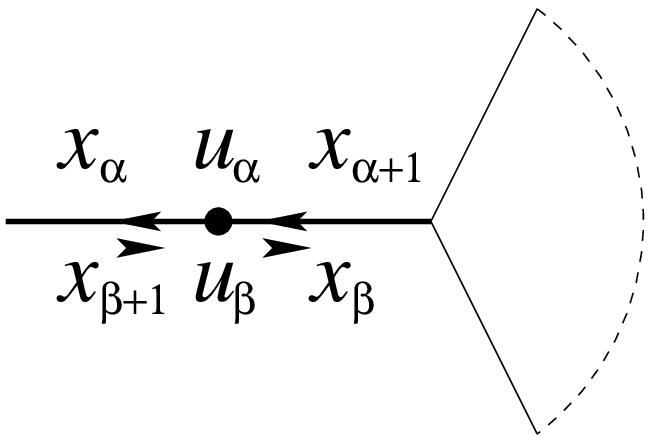}
\caption{Orientation patterns at a double vertex.}
\label{Double-vertex-pattern}
\end{figure}

An {\em Eulerian walk} is a balanced closed walk whose directions
have the same orientation on repeated edges of each fixed edge. A
midway-back avoided walk with positive sign is necessarily a
directed Eulerian walk and has no triple vertices. An Eulerian walk
with a direction is called a {\em directed Eulerian walk}.

An Eulerian walk $W$ is said to be {\em minimal} if there is no
Eulerian walk $W'$ such that $W'$ is contained properly in $W$ as
edge multisets, and said to be {\em elementary} if it is minimal and
there is no minimal Eulerian walk $W'$ such that $\supp W'$ is
properly contained in $\supp W$ as edge subsets. A minimal Eulerian
walk with a direction is called a {\em minimal directed Eulerian
walk}.

Let $(W,\varepsilon_W)$ be a directed closed walk of length $n$,
where $W$ is given in \eqref{Vertex-Edge-Seq}. The {\em
characteristic vector} of $(W,\varepsilon_W)$ on
$(\Sigma,\varepsilon)$ is a function $f_{(W,\,\varepsilon_W)}:
E\rightarrow{\Bbb Z}$ defined by
\begin{equation}\label{FW}
f_{(W,\,\varepsilon_W)}(x)=\sum_{x_i\in W,\,x_i=x}
[\varepsilon,\varepsilon_W](x_i).
\end{equation}
By Lemma~\ref{Walk-to-Flow}, $f_{(W,\,\varepsilon_W)}$ is an
integral flow on $(\Sigma,\varepsilon)$. Whenever
$\varepsilon_W=\varepsilon$ on $W$, we simply write
$f_{(W,\,\varepsilon_W)}$ as $f_W$.

Given a real-valued function $f$ on $E$. Let $\varepsilon_f$ be the
orientation on $\Sigma$ defined by
\begin{equation}\label{Orinetation-f}
\varepsilon_f(u,x)=\left\{\begin{array}{rl} -\varepsilon(u,x) &
\mbox{if $f(e)<0$, $x=uv$,}\\
\varepsilon(u,x) & \mbox{otherwise.}
\end{array}\right.
\end{equation}
It is trivial that $f$ is a flow on $(\Sigma,\varepsilon)$ if and
only if the absolute value function $|f|$ is a flow on
$(\Sigma,\varepsilon_f)$. Moreover,
$|f|=[\varepsilon,\varepsilon_f]\cdot f$.

A {\em cycle-tree} of $\Sigma$ is a connected signed subgraph $T$
which can be decomposed into edge-disjoint cycles $C_i$ (called {\em
block cycles}) and vertex-disjoint simple paths $P_j$ (called {\em
block paths}), denoted $T=\{C_i,P_j\}$, satisfying the four
conditions:
\begin{enumerate}
\item[(i)] $\{C_i\}$ is the collection of all cycles in $T$.

\item[(ii)] The intersection of two cycles is either empty or a single
vertex (called an {\em intersection vertex}).

\item[(iii)] Each $P_j$ intersects exactly two cycles and the
intersections are exactly the initial and terminal vertices of $P_j$
(also called {\em intersection vertices}).

\item[(iv)] Each intersection vertex is a {\em cut-point} (a vertex
whose removal increases the number of connected components of the
underline graph as a topological space of $1$-dimensional CW
complex), also known as a {\em separating vertex}
\cite[p.119]{Bondy-Murty1}.
\end{enumerate}
A cycle-tree is said to be {\em Eulerian} if it further satisfies
\begin{enumerate}
\item[(v)]
{\em Parity Condition}: Each balanced cycle has even number of
intersection vertices, while each unbalanced cycle has odd number of
intersection vertices.
\end{enumerate}

We call a block cycle in a cycle-tree to be an {\em end-block cycle}
if it has exactly one intersection vertex. The name cycle-tree is
justified as follows: if one converts each block cycle $C_i$ into a
vertex, each common intersection vertex of two block cycles into an
edge adjacent with the two vertices converted from the two block
cycles, and keep each block path $P_j$ connecting two vertices
converted from the two block cycles connected by $P_j$, then the
graph so obtained is indeed a tree.

An orientation $\varepsilon_T$ on a cycle-tree $T$ is called a {\em
direction} if $(T,\varepsilon_T)$ has neither sink nor source, and
for each block cycle $C$, the restriction $(C,\varepsilon_T)$ has
either sink or source at each cut-point of $T$ on $C$. We shall see
that $T$ admits a direction if and only if $T$ satisfies the Parity
Condition, and the direction is unique up to opposite sign. An
Eulerian cycle-tree $T$ with a direction $\varepsilon_T$ is called a
{\em directed Eulerian cycle-tree} $(T,\varepsilon_T)$. For
instance, the oriented signed graph given in
Figure~\ref{Cycle-Tree-Exmp} is an Eulerian cycle-tree with a
direction.
\begin{figure}[h]
\includegraphics[width=90mm]{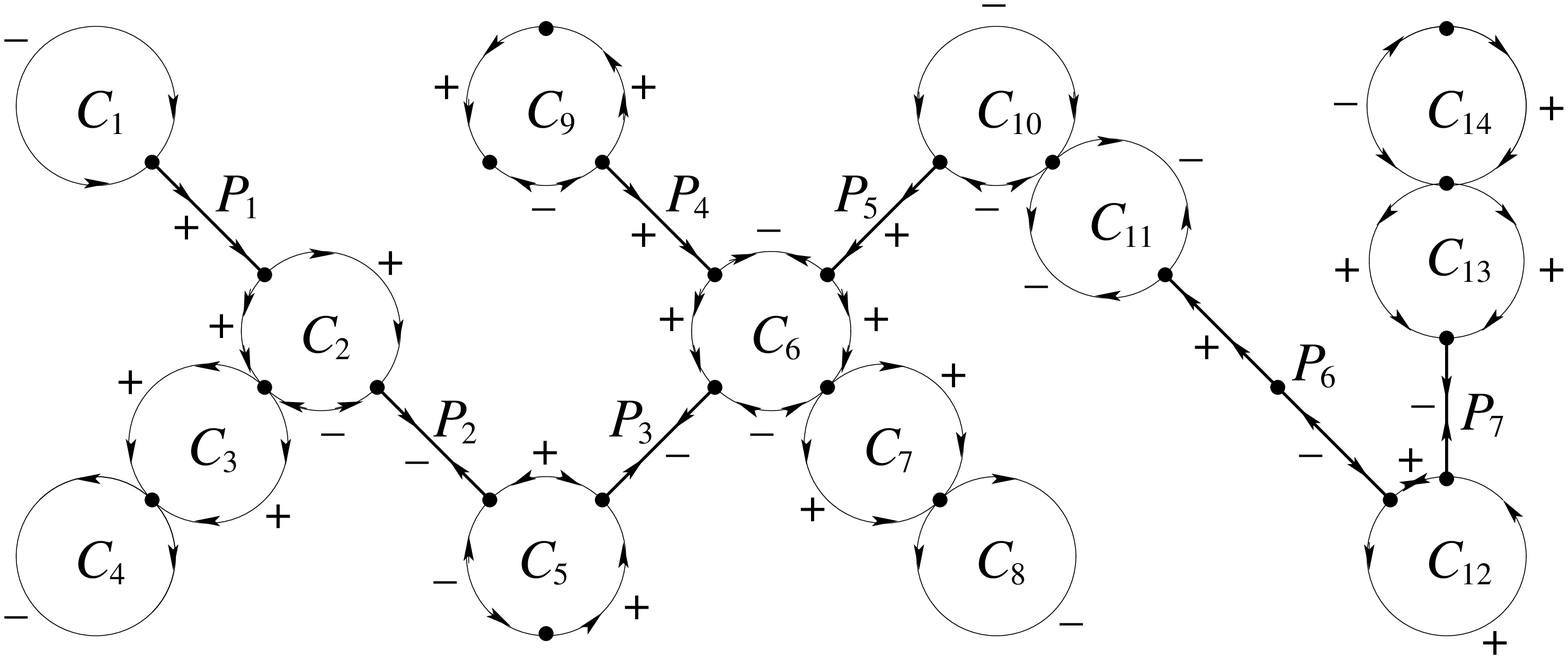}
\caption{An Eulerian cycle-tree and its direction.}
\label{Cycle-Tree-Exmp}
\end{figure}

Let $T=\{C_i,P_j\}$ be an Eulerian cycle-tree of $\Sigma$. The {\em
indicator} of $T$ is a function $I_T : E\rightarrow\Bbb Z$ defined
by
\begin{eqnarray}\label{Indicator-Function}
I_T(e)=\left\{
\begin{array}{ll}
1 & \hbox{if $e$ is on a block cycle of $T$,} \\
2 & \hbox{if $e$ is on a block path of $T$,}\\
0 & \hbox{otherwise.}
\end{array}
\right.
\end{eqnarray}
Given a direction $\varepsilon_T$ of $T$. Viewing both $(\Sigma,
\varepsilon)$ and $(T,\varepsilon_{T})$ as oriented signed subgraphs
of $\Sigma$, we have the coupling $[\varepsilon,\varepsilon_T]$. The
product function $[\varepsilon,\varepsilon_T]\cdot I_T$ determines a
vector in $\mathbb{Z}^E$ and is an integral flow of
$(\Sigma,\varepsilon)$ by
Theorem~\ref{Min-Directed-Eulerian-Walk-to-Eulerian-Cycle-Tree},
called the \emph{characteristic vector} of the directed Eulerian
cycle-tree $(T,\varepsilon_T)$ for $(\Sigma,\varepsilon)$.

An Eulerian cycle-tree is called a {\em (signed graph) circuit} if
it does not contain properly any Eulerian cycle-tree. We shall see
that each circuit $C$ must be one of the following three types.
\begin{itemize}
\item {\em
Type I}: $C$ consists of a single balanced cycle.

\item {\em Type II}: $C$ consists of two edge-disjoint unbalanced cycles $C_1,C_2$
having a single common vertex, written $C=C_1C_2$.

\item {\em Type III}: $C$ consists of two vertex-disjoint unbalanced cycles
$C_1,C_2$, and a simple path $P$ of positive length, such that
$C_1\cap P$ is the initial vertex and $C_2\cap P$ the terminal
vertex of $P$, written $C=C_1PC_2$.
\end{itemize}
The present definition of circuit seems different from that defined
in \cite{Zaslavsky1} and that adopted in \cite{Chen and Wang 1, Chen
and Wang 2}, but they are equivalent. The following characterization
of signed graph circuits shows the motivation of the
concept.\vspace{1ex}

\noindent \textsc{Characterization of Signed Graph Circuits.} Let
$f$ be a nontrivial integral flow of $(\Sigma,\varepsilon)$. Then
the following statements are equivalent.
\begin{enumerate}
\item[(a)] $f$ is elementary.

\item[(b)] $f$ is the characteristic vector of a directed circuit.

\item[(c)] There exists an elementary directed Eulerian
walk $(W,\varepsilon_W)$ such that
\[
f=f_{(W,\,\varepsilon_W)}.
\]
\end{enumerate}
\vspace{1ex}

\noindent {\bf Remark.} The characterization of signed graph
circuits was obtained by Bouchet \cite[p.283]{Bouchet}
(Corollary~2.3), using Zaslavsky's definition of circuits
\cite{Zaslavsky1}. As Zaslavsky pointed out himself, the central
observation of \cite[p.53]{Zaslavsky1} is the existence of a matroid
over the edge set of a signed graph whose circuits are exactly those
of Types I, II, III. Bouchet \cite{Bouchet} assumed (without
argument) that Zaslavsky's matroid is the same as the matroid whose
circuits are the supports of elementary flows. Indeed, it is trivial
to see that the circuits of the former are the circuits of the
latter. However, the converse seems not so obvious that no need
argument, though it is anticipated.
Corollary~\ref{Equiv-of-Elementary} implies that the converse is
indeed true. Now it is logically clear and aesthetically complete
that the matroid constructed by Zaslavsky \cite{Zaslavsky1} for a
signed graph is the same matroid whose circuits are the supports of
elementary chains (= elementary flows) of the signed graph in the
sense of Tutte \cite{Tutte3}; so are their dual matroids.

\vspace{1ex}

\noindent {\bf Main Theorem.} (Classification of Indecomposable
Integral Flows){\bf .} {\em Let $f$ be an integral flow of an
oriented signed graph $(\Sigma,\varepsilon)$.

(a) Then $f$ is indecomposable if and only if there exists an
Eulerian cycle-tree $T$ such that
\[
f=[\varepsilon,\varepsilon_f]\cdot I_T.
\]

(b) If $T$ is an Eulerian cycle-tree other than a circuit, then for
each closed walk $W$ of minimum length that uses all edges of $T$,
there is a decomposition
\[
W=C_0P_1C_1\cdots P_kC_kP_{k+1}, \sp k\geq 1,
\]
where $C_i$ are entire end-block cycles of $T$ and
$C_{i}P_{i+1}C_{i+1}$ are circuits of Type III with $C_{k+1}=C_0$,
such that
\[
I_T=\frac{1}{2}\sum_{i=0}^k I_{\Sigma(C_{i}P_{i+1}C_{i+1})}.
\]}

\section{Flow Reduction Algorithm}

The decomposability of an integral flow $f$ on
$(\Sigma,\varepsilon)$ is equivalent to the decomposability of the
flow $|f|$ on $(\Sigma,\varepsilon_f)$. So without loss of
generality, to decompose an integral flow, one only needs to
consider nonnegative nontrivial integral flows of
$(\Sigma,\varepsilon)$.
The following {\em Flow Reduction Algorithm} ({\em FRA}) finds
explicitly a minimal directed Eulerian walk from a given nontrivial
integral flow.

Let us first show that the characteristic vector of a directed
closed walk is an integral flow.

\begin{lemma}\label{Walk-to-Flow}
Let $(W,\varepsilon_W)$ be a directed closed walk. Then the function
$f_{(W,\,\varepsilon_W)}$ defined by \eqref{FW} is an integral flow
of $(\Sigma,\varepsilon)$.
\end{lemma}
\begin{proof} Let $W=u_0x_1u_1x_2\ldots x_nu_{n}$ be the vertex-edge sequence,
where each edge $x_i$ has end-vertices $u_{i-1},u_{i}$, and
$u_{n}=u_0$. Fix a vertex $v\in V$,
let $u_{a_1},u_{a_2},\ldots,u_{a_k}$ be the
sequence that $v$ appears in $W$. Since
$\varepsilon_W(u_{a_j},x_{a_j})
+\varepsilon_W(u_{a_j},x_{a_j+1})=0$, we have
\begin{eqnarray*}
\partial f_{(W,\varepsilon_W)} (v) &=& \sum_{x\in E,\,u\in{\rm
End}(x),\,u=v}
\varepsilon(u,x)f_{(W,\varepsilon_W)}(x) \\
&=& \sum_{x\in E,\,u\in{\rm End}(x),\,u=v \atop x_i\in W,\, x_i=x}
\varepsilon(u,x) [\varepsilon,
\varepsilon_W](x_i)\\
&=& \sum_{x_i\in W,\, u\in{\rm End}(x_i),\, u=v} \varepsilon_W(u,x_i) \\
&=& \sum_{j=1}^k \big[\varepsilon_W(u_{a_j},x_{a_j})
+\varepsilon_W(u_{a_j},x_{a_j+1})\big] = 0.
\end{eqnarray*}
Hence the function $f_{(W,\,\varepsilon_W)}$ is an integral flow of
$(\Sigma,\varepsilon)$.
\end{proof}

\noindent \textbf{Flow Reduction Algorithm (FRA).} Given a
nontrivial integral flow $f$ on $(\Sigma,\varepsilon)$.

{\sc Step~0.} {\it Choose an edge $x_1$ in $\supp f$ with
end-vertices $u_0,u_1$. Initiate a half-closed and half-open walk
$u_0x_1$. Set $W:=u_0x_1$ and $\ell:=1$. Go to {\sc Step~1}.}
\vspace{1ex}

{\sc Step~1.} {\it If $u_\ell\not\in W$, go to {\sc Step~2}. If
$u_\ell\in W$, say, $u_\ell=u_\beta$ with the greatest index
$\beta<\ell$, go to {\sc Step~3}.}\vspace{1ex}

{\sc Step~2.} {\it There exists an edge $x_{\ell+1}$ in $\supp f'$
other than $x_\ell$, where $f':=f-f_{(W,\,\varepsilon_f)}$, having
end-vertices $u_\ell,u_{\ell+1}$ such that
$\varepsilon_f(u_\ell,x_{\ell+1})=-\varepsilon_f(u_\ell,x_\ell)$.
Set $W:=Wu_\ell x_{\ell+1}$ and $\ell:=\ell+1$. Return to {\sc
Step~1}.}\vspace{1ex}

{\sc Step~3.} {\it If $u_\beta$ is a double point of $W$, say,
$u_\alpha=u_\beta$ with $\alpha<\beta<\ell$, {\sc Stop}. For the
case $\varepsilon_f(u_\ell,x_\ell) =
-\varepsilon_f(u_\beta,x_{\beta+1})$, set
\begin{equation}\label{beta-ell}
W:=u_\beta x_{\beta+1}u_{\beta+1}\ldots u_{\ell-1}x_\ell u_\ell.
\end{equation}
For the case $\varepsilon_f(u_\ell,x_\ell)
=\varepsilon_f(u_\beta,x_{\beta+1})$, set
\begin{equation}\label{alpha-ell}
W:=u_\alpha x_{\alpha+1}u_{\alpha+1}\ldots u_{\ell-1}x_\ell u_\ell.
\end{equation}
Then $(W,\varepsilon_f)$ is a {\rm directed Eulerian walk}. If
$u_\beta$ is a single point of $W$, go to {\sc Step~4}.}

{\sc Step~4.} {\it If there exist double vertices
$u_\alpha,u_\gamma$ in $W$ with $\alpha<\beta<\gamma$ such that
$u_\alpha=u_\gamma$, {\sc Stop}. For the case
$\varepsilon_f(u_\ell,x_\ell) =-\varepsilon_f(u_\beta,x_{\beta+1})$,
set $W$ to be of (\ref{beta-ell}). For the case
$\varepsilon_f(u_\ell,x_\ell) = \varepsilon_f(u_\beta,x_{\beta+1})$,
set
\begin{equation}\label{beta-alpha-gamma-ell}
W:=u_\beta x_{\beta} u_{\beta-1} \ldots u_{\alpha+1}
x_{\alpha+1}u_\alpha (u_\gamma) x_{\gamma+1} u_{\gamma+1} \ldots
u_{\ell-1}x_\ell u_\ell.
\end{equation}
Then $(W,\varepsilon_f)$ is a {\rm directed Eulerian walk}.
Otherwise, go to {\sc Step~5}.}

{\sc Step~5.} {\it If $\varepsilon_f(u_\ell,x_\ell) = -\varepsilon_f
(u_\beta,x_{\beta+1})$, {\sc stop}. Set $W$ to be of
(\ref{beta-ell}). Then $(W,\varepsilon_f)$ is a directed Eulerian
walk. If $\varepsilon_f(u_\ell,x_\ell) =
\varepsilon_f(u_\beta,x_{\beta+1})$, return to {\sc \sc
Step~2}.}\vspace{2ex}

It is clear from {\sc Step~3} that the multiplicity of each vertex
in the closed walk $W$ (obtained by FRA) is at most two. So $W$ has
only possible double vertices and possible double edges. At each
double vertex of $W$, say $u_\alpha=u_\beta$ with $\alpha<\beta$,
{\sc Step~5} implies
\begin{equation}\label{Double-Point-Pattern}
\varepsilon(u_\beta,x_\beta) = \varepsilon_W(u_\alpha,x_{\alpha+1})
= -\varepsilon_W(u_\alpha,x_{\alpha}) =
-\varepsilon_W(u_\beta,x_{\beta+1});
\end{equation}
see Figure~\ref{Double-vertex-pattern}. It is possible that
$(u_\beta,x_{\beta+1})=(u_\alpha,x_\alpha)$; if so, the repeated
edges $x_\alpha,x_{\beta+1}$ have the same orientation in
$(W,\varepsilon_W)$. This means that $(W,\varepsilon_W)$ is a
directed Eulerian walk.

In {\sc Step~2}, both functions $f_W,f'$ are not conservative at
$u_\ell$. In fact, if $u_\ell\neq u_0$, then $\partial
f_{(W,\,\varepsilon_f)}(u_\ell)=\varepsilon_f(u_\ell,x_\ell)$ and
$\partial f'(u_\ell)=-\varepsilon_f(u_\ell,x_\ell)$; if
$u_\ell=u_0$, then $\partial f_{(W,\,\varepsilon_f)}
(u_\ell)=2\varepsilon_f(u_\ell,x_\ell) =2\varepsilon_f(u_0,x_1)$.
Hence $\partial f_{(W,\,\varepsilon_f)}(u_\ell)\neq 0$. This means
that there exists an edge $x_{\ell+1}$ in $\supp f'$ at $u_\ell$
such that $\varepsilon_f(u_\ell,x_{\ell+1})
=-\varepsilon_f(u_\ell,x_\ell)$. Thus the length of $W$ is increased
by one. Since the multiset $(E,|f|)$ is finite, FRA stops with a
directed closed walk $(W,\varepsilon_f)$. Moreover, {\sc Step~4}
implies that each double vertex in $W$ (obtained by FRA) is a
cut-point of $\Sigma(W)$.

\begin{lemma}\label{W-W-Inverse}
Let $W$ be a directed walk. 
Then FRA finds no directed closed walk along $W$ if and only if FRA
finds no directed closed walk along $W^{-1}$.
\end{lemma}
\begin{proof}
This seems to be quite trivial. In fact, let FRA find a directed
closed walk along $W$. Then $W$ contains one of the three patterns
of closed walks: (a), (b) and (c) in Figure~\ref{Three-patterns},
where $\alpha<\beta<\gamma<\delta$, $\varepsilon_W(u_\beta,x_\beta)
= -\varepsilon_W(u_\alpha,x_{\alpha+1})$ in (a),
$\varepsilon_W(u_\gamma,x_\gamma) =
-\varepsilon_W(u_\alpha,x_{\alpha+1})$ in (b), and
$\varepsilon_W(u_\delta,x_\delta) =
-\varepsilon_W(u_\alpha,x_{\alpha+1})$ in (c). The reversions of
patterns (a), (b) and (c), as subwalks in $W^{-1}$, have the same
patterns as (a), (b) and (c) respectively. The subwalks from
$u_\alpha$ to $u_\beta$ in (a), (b), (c) may contain some double
vertices and double edges; so do the subwalks from $u_\beta$ to
$u_\gamma$ in (b) and (c); so does the subwalk from $u_\gamma$ to
$u_\delta$ in (c).

Note that when FRA is applied to $W$, the algorithm may stop and
find a directed closed walk before it reaches $u_\beta$ in (a), or
before it reaches $u_\gamma$ in (b), or before it reaches $u_\delta$
in (c). If so, when FRA is applied to $W^{-1}$, the algorithm stops
and finds a directed closed walk along $W^{-1}$ before it reaches
$u_{\delta^{-1}}$, or $u_{\gamma^{-1}}$, or $u_{\beta^{-1}}$, or
$u_{\alpha^{-1}}$. If not, when FRA is applied to $W^{-1}$, the
algorithm stops and finds a directed closed walk when it reaches
$u_{\alpha^{-1}}$. This means that FRA finds a directed closed walk
along $W^{-1}$.
\begin{figure}[h]
\subfigure[]{\includegraphics[height=40mm]{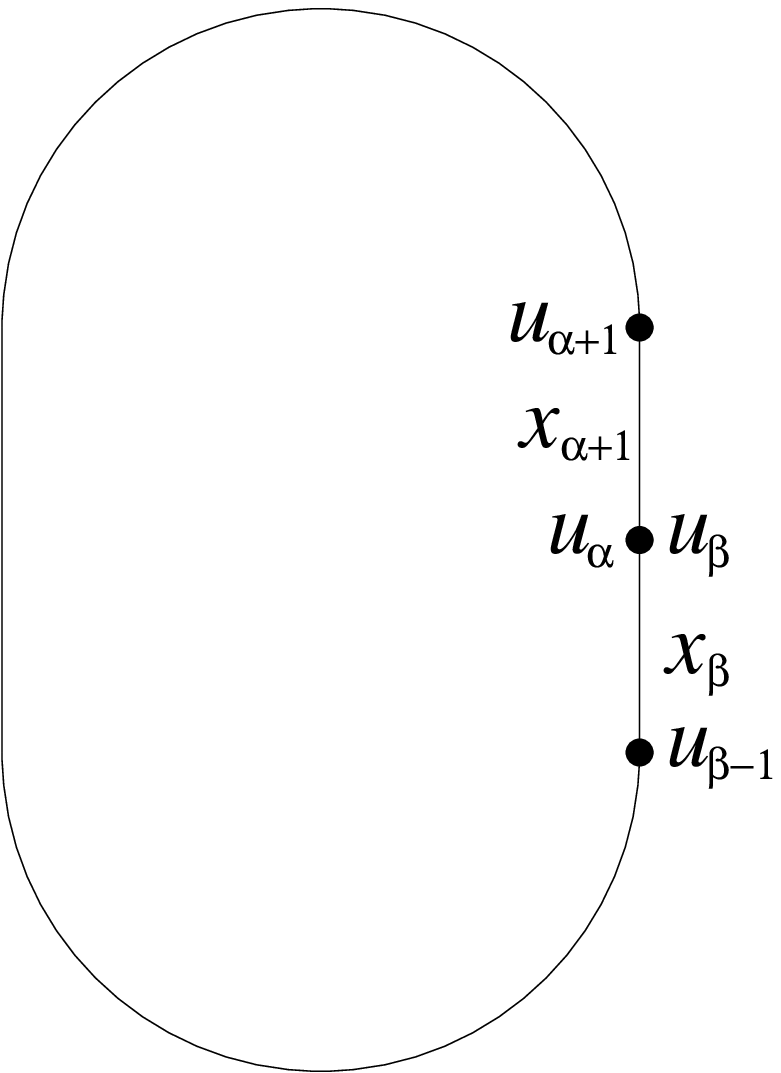}}\hspace{5mm}
\subfigure[]{\includegraphics[height=40mm]{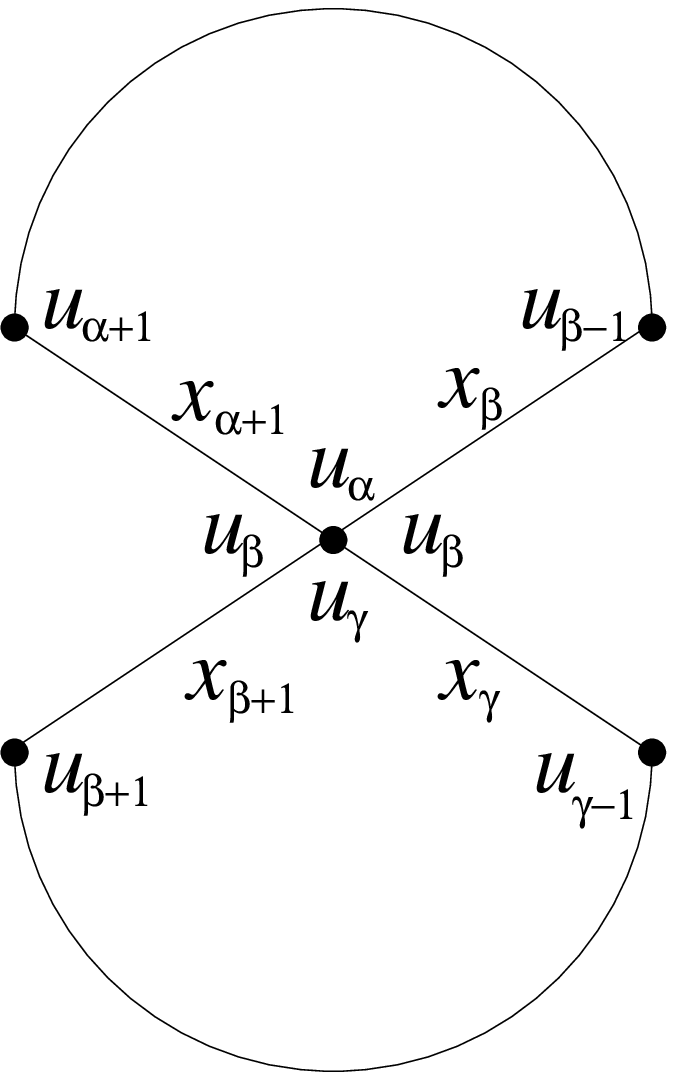}}\hspace{5mm}
\subfigure[]{\includegraphics[height=40mm]{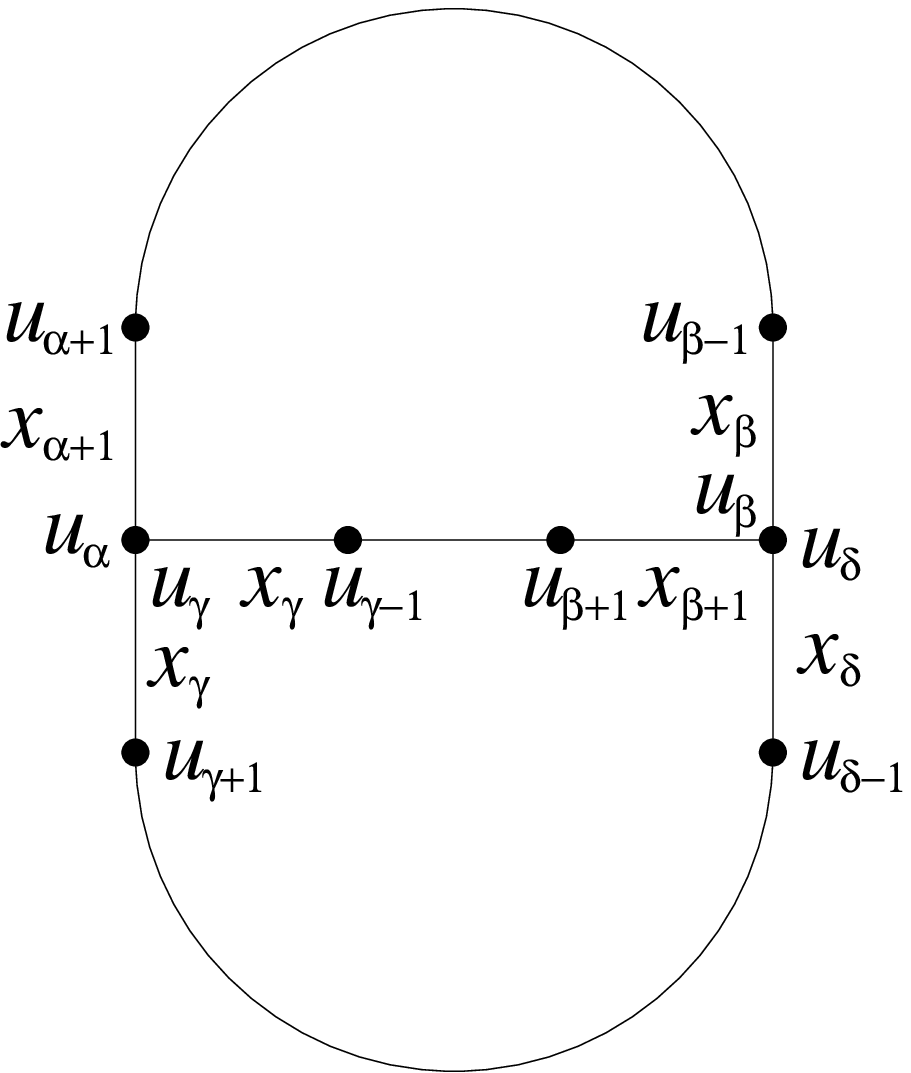}}
\caption{Three patterns that FRA stops.} \label{Three-patterns}
\end{figure}

Conversely, let FRA find a directed closed walk along $W^{-1}$. Then
FRA finds a directed closed walk along $(W^{-1})^{-1}$, which is
exactly the walk $W$. We have seen that FRA finds a directed closed
walk along $W$ if and only if FRA finds a directed closed walk along
$W^{-1}$.
\end{proof}

\begin{lemma}\label{Midway-to-Eulerian-No-Triple}
Let $(W,\varepsilon_W)$ be a directed midway-back avoided walk. Then
\begin{enumerate}
\item[(a)] $W$ has no triple vertices, that is, the multiplicity of
each vertex and of each edge in $W$ is at most two.

\item[(b)] $(W,\varepsilon_W)$ is a directed Eulerian walk.
\end{enumerate}
\end{lemma}
\begin{proof} Write $W=u_0x_1u_1x_2\cdots u_{n-1}x_nu_n$.

(a) Suppose there is a vertex appeared three times in $W$, say,
$u_\alpha=u_\beta=u_\gamma$ with $\alpha<\beta<\gamma$; see
Figure~\ref{triple-point}.
\begin{figure}[h]
\includegraphics[height=35mm]{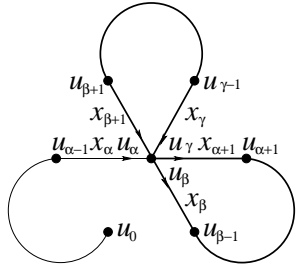}
\caption{The pattern of a triple point.}\label{triple-point}
\end{figure}
Since $(W,\varepsilon_W)$ is midway-back avoided, we have
\[
\varepsilon_W(u_\beta,x_\beta) =
\varepsilon_W(u_\alpha,x_{\alpha+1}),
\]
\[
\varepsilon_W(u_\gamma,x_\gamma) =
\varepsilon_W(u_\alpha,x_{\alpha+1}),
\]
\[
\varepsilon_W(u_\gamma,x_\gamma) =
\varepsilon_W(u_\beta,x_{\beta+1}).
\]
Then
\[
\varepsilon_W(u_\gamma,x_\gamma) = -\varepsilon_W(u_\beta,x_{\beta})
= - \varepsilon_W(u_\alpha,x_{\alpha+1}) = -
\varepsilon_W(u_\gamma,x_\gamma),
\]
which is a contradiction.

(b) Let $u_\alpha=u_\beta$ with $\alpha<\beta$ and let $x_{\beta+1}
(=x_\alpha$) be a repeated edge. Then $u_{\beta+1} = u_{\alpha-1}$.
Suppose $\varepsilon_W(u_\beta,x_{\beta+1}) = -
\varepsilon_W(u_\alpha,x_\alpha)$. Then
\[
\varepsilon_W(u_{\beta+1},x_{\beta+1}) = -
\varepsilon_W(u_{\alpha-1},x_\alpha).
\]
This means that $(W,\varepsilon_W)$ is midway-back at
$u_{\alpha-1}$, which is a contradiction. So
$\varepsilon_W(u_\beta,x_{\beta+1}) =
\varepsilon_W(u_\alpha,x_\alpha)$. This means that $\varepsilon_W$
has the same orientation on repeated edges. Thus $(W,\varepsilon_W)$
is a directed Eulerian walk.
\end{proof}

\begin{lemma}\label{FRA-Walk-to-Midway-Cut-Point}
Let $(W,\varepsilon_f)$ be a directed closed walk found by FRA. Then

{\rm (a)} $(W,\varepsilon_f)$ is midway-back avoided.

{\rm (b)} Each double vertex in $W$ is a cut-point of $\Sigma(W)$.
\end{lemma}
\begin{proof}
(a) The directed walk $(W,\varepsilon_f)$ satisfies
(\ref{Double-Point-Pattern}). By definition $(W,\varepsilon_f)$ is
midway-back avoided.

(b) Assume that FRA stops at time $\ell$ and finds a directed closed
walk $(W,\varepsilon_f)$, but did not stop before $\ell$. The two
forms (\ref{beta-ell}) and (\ref{alpha-ell}) of $W$ are the same
kind, having indices increasing. However, the form
(\ref{beta-alpha-gamma-ell}) of $W$ is special; its indices from
$u_\beta$ to $u_\alpha$ are decreasing. We may reduce the form
(\ref{beta-alpha-gamma-ell}) of $W$ to the form whose indices are
increasing as follows.

Consider the directed walk $(W',\varepsilon_f)$, where $W'=W_1W_2$,
\begin{align*}
W_1 &= u_\alpha x_{\alpha+1}u_{\alpha+1}\ldots u_\beta
x_{\beta+1}u_{\beta+1} \ldots u_{\gamma-1} x_\gamma u_\gamma,
\\
W_2 &= u_{\gamma} x_{\gamma+1} u_{\gamma+1}\ldots u_{\ell-1} x_\ell
u_\ell, \sp u_\gamma=u_\alpha,\; u_\ell=u_\beta.
\end{align*}
Applying FRA to $W_1W_2$, the algorithm cannot stop before $\ell$,
but stops at time $\ell$ and finds the directed closed walk $W$. Of
course, FRA finds no directed closed walk along $W_1$. Writing
$W_1^{-1}$ in increasing-order of indices and applying FRA to
$W_1^{-1}$, by Lemma~\ref{W-W-Inverse} the algorithm finds no
directed closed walk along $W_1^{-1}$. Now applying FRA to
$W_1^{-1}W_2$, the algorithm cannot stop before $\ell$, but stops at
time $\ell$ and finds the same directed closed walk $W$, having
indices increasing.

Without loss of generality we may assume that $(W,\varepsilon_f)$
(obtained by FRA) has the form
\begin{equation}\label{W-0-1}
W=u_0x_1u_1x_2\ldots u_{\ell-1}x_\ell u_\ell, \sp u_0=u_\ell.
\end{equation}
Suppose $W$ has a double vertex $u$ that is not a cut-point of
$\Sigma(W)$, say, $u=u_\delta=u_\eta$ with $\delta<\eta$. Then there
exist vertices $u_\mu,u_\nu$ in $W$ such that $u_\mu=u_\nu$, where
$\delta<\mu<\eta$ and either $\eta<\nu$ or $\nu<\delta$. With the
indices modulo $\ell$, the closed walk $W$ can be written as the
form (see Figure~\ref{n-cut-point})
\[
W=u_\delta x_{\delta+1}u_{\delta+1}\dots x_{\mu}u_\mu
x_{\mu+1}\ldots x_{\eta}u_\eta x_{\eta+1} \ldots x_{\nu}u_{\nu}
x_{\nu+1}\ldots u_{\delta-1}x_{\delta}u_\delta.
\]
Consider the case $\delta<\mu<\eta<\nu$. If $\nu<\ell$, FRA stops in
{\sc Step~4} at time $\nu$ and finds the directed closed walk
\[
u_\mu x_\mu u_{\mu-1}\ldots u_{\delta+1}x_{\delta+1}u_\delta
(u_\eta)x_{\eta+1}u_{\eta+1}\ldots u_{\nu-1}x_{\nu}u_\nu
\]
in Figure~\ref{n-cut-point}; this is a contradiction. If $\nu=\ell$,
then $u_{\nu}x_{\nu+1}u_{\nu+1}=u_0x_1u_1$, FRA stops in {\sc Step
4} at time $\eta$ and finds the directed closed walk
\[
u_\nu x_{\nu+1} u_{\nu+1}\ldots u_{\delta-1}x_{\delta}u_\delta
(u_\eta)x_{\eta}u_{\eta-1}\ldots u_{\mu+1}x_{\mu+1}u_\mu
\]
in Figure~\ref{n-cut-point}; this is a contradiction. For the case
$\nu<\delta<\mu<\eta$, it is analogous to the case
$\delta<\mu<\eta<\nu$.
\begin{figure}[h]
\includegraphics[height=40mm]{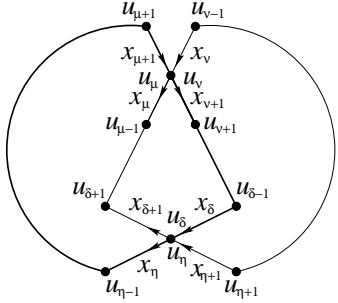}
\caption{A double vertex that is not a cut-point.}
\label{n-cut-point}
\end{figure}
\end{proof}

\begin{thm}\label{Midway-Cut-Point-to-Eulerian-Cycle-Tree}
Let $(W,\varepsilon_W)$ be a directed closed walk such that
\begin{enumerate}
\item[(i)] $(W,\varepsilon_W)$ is a directed midway-back avoided walk;

\item[(ii)] each double vertex in $W$ is a cut-point of $\Sigma(W)$.
\end{enumerate}
Then $\Sigma(W)$ is an Eulerian cycle-tree, the restriction of
$\varepsilon_W$ to $\Sigma(W)$ is a direction on the cycle-tree, and
$W$ uses each edge of block cycles once and each edge of block paths
twice, crossing from one block to the other block at each cut-point.
\end{thm}
\begin{proof}
Lemma~\ref{Midway-to-Eulerian-No-Triple} implies that $W$ has no
triple vertices, that is, $W$ has only possible double vertices and
double edges. Since each double vertex in $W$ is a cut-point of
$\Sigma(W)$, then each connected component of the signed subgraph
induced by the double edges of $W$ is a simple path (of possible
zero length), called a {\em double-edge path} of $W$. Remove the
internal part of each double-edge path of positive length from $W$,
we obtain an Eulerian graph whose vertex degrees are either 2 or 4.
The Eulerian graph can be decomposed into edge-disjoint cycles,
called {\em block cycles}. Each double-edge path (of possible zero
length) connects exactly two block cycles. Since $\Sigma(W)$ is
connected and each double vertex in $W$ is a cut-point of
$\Sigma(W)$, it follows that $\Sigma(W)$ is a cycle-tree.

It is clear that $W$ uses each edge of block cycles once and each
edge of block paths twice, and crosses from one block to the other
block at each cut-point. Since $(W,\varepsilon_W)$ is midway-back
avoided, it follows that the restriction of $\varepsilon_W$ to
$\Sigma(W)$ is a direction on the cycle-tree. Thus $\Sigma(W)$ is an
Eulerian cycle-tree by Lemma~\ref{Parity-Consition-via-Direction}.
\end{proof}

\begin{cor}\label{FRA-Walk-to-Eulerian-Cycle-Tree}
Let $(W,\varepsilon_f)$ be a directed closed walk found by FRA. Then
$\Sigma(W)$ is an Eulerian cycle-tree with direction
$\varepsilon_f$, and $W$ uses each edge of block cycles once and
each edge of block paths twice, crossing from one block to the other
block at each cut-point.
\end{cor}
\begin{proof}
It follows from Lemma~\ref{FRA-Walk-to-Midway-Cut-Point} and
Theorem~\ref{Midway-Cut-Point-to-Eulerian-Cycle-Tree}.
\end{proof}

\begin{thm}[Flow Reduction Theorem]\label{Flow-Reduction-Theorem}
Let $f$ be a nontrivial integral flow of $(\Sigma,\varepsilon)$.
Then there exist minimal directed Eulerian walks
$(W_i,\varepsilon_{f})$ and Eulerian cycle-trees $T_i=\Sigma(W_i)$
such that
\begin{equation}
f=\sum f_{(W_i,\,\varepsilon_{f})} = \sum
[\varepsilon,\varepsilon_f]\cdot I_{T_i},
\end{equation}
where $f_{(W_i,\,\varepsilon_f)}$ are given by \eqref{FW} and
$I_{T_i}$ by \eqref{Indicator-Function}.

Furthermore, if $f$ is indecomposable, then there exists a minimal
directed Eulerian walk $(W,\varepsilon_f)$ and an Eulerian
cycle-tree $T=\Sigma(W)$ such that
\begin{equation}
f=f_{(W,\,\varepsilon_f)}=[\varepsilon,\varepsilon_f]\cdot I_T.
\end{equation}
\end{thm}
\begin{proof}
Consider the nonnegative integral flow $|f|$ of
$(\Sigma,\varepsilon_f)$, where $\varepsilon_f$ is defined by
\eqref{Orinetation-f}. Let $\Sigma(f)$ denote the signed subgraph
induced by the edge subset $\supp f$. Then FRA finds a directed
Eulerian walk $(W_1,\varepsilon_f)$ on the oriented signed graph
$(\Sigma(f),\varepsilon_f)$, such that $f_{W_1}\leq |f|$, where
$f_{W_1}$ is given by \eqref{FW} with $\varepsilon=\varepsilon_f$.
Corollary~\ref{FRA-Walk-to-Eulerian-Cycle-Tree} implies that
$\Sigma(W_1)$ is an Eulerian cycle-tree $T_1$, and
Theorem~\ref{Eulerian-Cycle-Tree-to-Minimality} implies that $W_1$
is a minimal Eulerian walk. Then
Theorem~\ref{Min-Directed-Eulerian-Walk-to-Eulerian-Cycle-Tree}
implies $f_{W_1}=I_{T_1}$, the indicator function of $T_1$ defined
by \eqref{Indicator-Function}.

If $f_1:=|f|-f_{W_1}\neq 0$, then FRA finds a minimal directed
Eulerian walk $(W_2,\varepsilon_f)$ on
$(\Sigma(f_{1}),\varepsilon_f)$, such that $f_{W_2}\leq |f|-f_{W_1}$
and $f_{W_2}=I_{T_2}$, where $T_2$ is the Eulerian cycle-tree
$\Sigma(W_2)$. Likewise, if $f_2:=|f|-f_{W_1}-f_{W_2}\neq 0$, then
FRA finds a minimal directed Eulerian walk $(W_3,\varepsilon_f)$ on
$(\Sigma(f_{2}),\varepsilon_f)$, such that $f_{W_3}\leq
|f|-f_{W_1}-f_{W_2}$ and $f_{W_3}=I_{T_3}$, where $T_3$ is the
Eulerian cycle-tree $\Sigma(W_3)$. Continue this procedure, we
obtain minimal directed Eulerian walks
\[
(W_1,\varepsilon_f),\; (W_2,\varepsilon_f),\; \ldots,\;
(W_k,\varepsilon_f)
\]
on $(\Sigma(f),\varepsilon_f)$, such that $|f|=\sum_{i=1}^k
f_{W_i}=\sum_{i=1}^k I_{T_i}$, where $T_i$ are the Eulerian
cycle-trees $\Sigma(W_i)$ and $f_{W_i}=I_{T_i}$. Note that
\[
f=[\varepsilon,\varepsilon_f]\cdot |f|, \sp
f_{(W_i,\,\varepsilon_f)} =[\varepsilon,\varepsilon_f]\cdot f_{W_i}.
\]
We obtain $f =\sum_{i=1}^k f_{(W_i,\,\varepsilon_f)} =\sum_{i=1}^k
[\varepsilon,\varepsilon_f]\cdot I_{T_i}$.

If $f$ is indecomposable, by definition we must have $k=1$.
\end{proof}

\section{Characterizations of Eulerian Cycle-trees}

This section is to establish properties satisfied by Eulerian
cycle-trees such as the Existence and Uniqueness of Direction, the
Minimality, and the Half-integer Scale Decomposition. We shall see
the equivalence of indecomposable flows, minimal Eulerian walks, and
Eulerian cycle-trees. The byproduct is the equivalence of circuits,
elementary flows, and elementary Eulerian walks, and the
classification of circuits.

\begin{lemma}\label{Directed-Walk-to-Sign}
Let $W=u_0x_1u_1x_2\ldots u_{n-1}x_nu_n$ be a walk with a direction
$\varepsilon_W$. Then
\[
\varepsilon_W(u_n,x_n)=-\sigma(W)\varepsilon_W(u_0,x_1).
\]
In particular, if $(W,\varepsilon_W)$ is a directed closed walk,
then $\sigma(W)=1$.
\end{lemma}
\begin{proof}
The direction $\varepsilon_W$ must be constructed inductively as
follows:
\[
\varepsilon_W(u_i,x_i)=-\sigma(x_i)\varepsilon_W(u_{i-1},x_i),
\]
\[
\varepsilon_W(u_i,x_{i+1})=-\varepsilon_W(u_i,x_i), \sp 1\leq i\leq
n.
\]
Then $\varepsilon_W(u_n,x_n)$ is determined by
$\varepsilon_W(u_0,x_1)$ as $\varepsilon_W(u_n,x_n) =-
\sigma(W)\varepsilon_W(u_0,x_1)$.

If $(W,\varepsilon_W)$ is a directed closed walk, then
$\varepsilon_W(u_n,x_n)=-\varepsilon_W(u_0,x_1)$. It is clear that
$\sigma(W)=1$.
\end{proof}

Let $T=\{C_i,P_j\}$ be a cycle-tree throughout. We choose a block
cycle $C_0$ and write
\begin{align}\label{Closed-Walk-C_0}
C_0 &=u_0x_1u_1x_2\ldots u_{l-1}x_lu_l, \sp u_l=u_0.
\end{align}
If $T$ has two or more block cycles, we require $C_0$ to an
end-block cycle, having $u_0$ as its unqiue intersection vertex. Let
$P$ be the block path (of possible zero length) from the vertex
$u_0$ on $C_0$ to a vertex $w_0$ on another block cycle $C_1$. We
write
\begin{align}
P &=v_0y_1v_1y_2\ldots v_{m-1}y_mv_m, \sp v_m=w_0.
\end{align}
Remove the cycle $C_0$ and the internal part of the path $P$, we
obtain a cycle-tree
\begin{equation}\label{T1}
T_1=T\backslash (C_0\cup P),
\end{equation}
which has one fewer block cycle than $T$. Choose an edge $z_1$ on
$C_1$ incident with $w_0$ and switch the sign of $z_1$, we obtain a
cycle-tree $T_1'$. If $T$ is Eulerian, so is $T'_1$, for the block
cycle $C_1$ has one fewer intersection vertex in $T'_1$ than in $T$.
This procedure will be recalled in the proof of the following
Lemma~\ref{Cycle-Tree-to-Minimum-Closed-Walk} and
Theorem~\ref{Parity-Consition-via-Direction}.

\begin{lemma}\label{Cycle-Tree-to-Minimum-Closed-Walk}
Let $T$ be a cycle-tree. Then there exists a closed walk $W$ on $T$
that uses each edge of block cycles once and each edge of block
paths twice, and crosses from one block to the other block at each
cut-point.

Moreover, each such $W$ is a closed walk of minimum length that uses
all edges of $T$, and vice versa.
\end{lemma}
\begin{proof}
If $T$ has only one block cycle, then $T$ is a cycle $C_0$ and can
be written as a closed walk in (\ref{Closed-Walk-C_0}). If $T$ has
two or more block cycles, then by induction there is a closed walk
$W_1$ on $T_1$ in (\ref{T1}) such that $W_1$ crosses from one block
to the other block at each intersection vertex. Then
$W=C_0PW_1P^{-1}$ is the required closed walk on $T$; see
Figure~\ref{Cycle-Tree}. The property of minimum length is trivial.
\begin{figure}[h]
\centering\includegraphics[width=70mm]{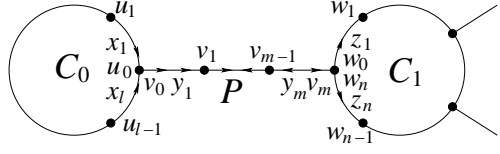}\\
\caption{End-block cycle of a cycle-tree $T$.} \label{Cycle-Tree}
\end{figure}
\end{proof}

\begin{thm}[Existence and Uniqueness of Direction on Eulerian Cycle-Tree]
\label{Parity-Consition-via-Direction}
Let $T$ be a cycle-tree. Then $T$ satisfies the Parity Condition if
and only if there exists a unique direction $\varepsilon_T$ on $T$
up to opposite sign.
\end{thm}
\begin{proof}
``$\Rightarrow$'': We proceed by induction on the number of block
cycles of $T$. When $T$ has only one block cycle, then $T$ is the
cycle itself, and the cycle has to be balanced. It is clear that a
balanced cycle has a unique direction up to opposite sign.

Assume that $T$ has two or more block cycles. Then $T_1$ in
\eqref{T1} is a cycle-tree having one fewer block cycle than $T$.
Switch the sign of the edge $z_1$ in $T_1$, we obtain an Eulerian
cycle-tree $T'_1$. By induction there exists a unique direction
$\varepsilon_{T'_1}$ (up to opposite sign) on $T'_1$. Let us switch
the sign of $z_1$ in $T'_1$ back to the sign of $z_1$ in $T_1$ and
define an orientation $\varepsilon_{T_1}$ on $T_1$ by setting
$\varepsilon_{T_1}=\varepsilon_{T'_1}$ for all vertex-edge pairs
except
\[
\varepsilon_{T_1}(w_0,z_1) =-\varepsilon_{T'_1}(w_0,z_1).
\]
Then $(C_1,\varepsilon_{T_1})$ has either a sink or a source at
$w_0$. Let $\varepsilon_P$ be a direction on $P$ such that
$\varepsilon_P(v_m,y_m)=-\varepsilon_{T_1}(w_0,z_1)$, and
$\varepsilon_{C_0}$ be a direction on $C_0$ such that
$\varepsilon_{C_0}(u_l,x_l)=-\varepsilon_P(v_0,y_1)$. Then the joint
orientation $\varepsilon_{C_0}\vee \varepsilon_{P}\vee
\varepsilon_{T_1}$ gives rise to a direction $\varepsilon_T$ on $T$;
see Figure~\ref{Cycle-Tree}.

Let $\varepsilon'_T$ be an arbitrary direction on $T$. Then
$\varepsilon'_T$ induces directions $\varepsilon'_{C_0},
\varepsilon'_{P}, \varepsilon'_{T'_1}$ on $C_0, P, T'_1$
respectively, where
\begin{align*}
\varepsilon'_{T'_1}(w_0,z_1) &= -\varepsilon'_T(w_0,z_1),\\
\varepsilon'_{P}(v_m,y_m)    &= -\varepsilon_T(w_0,z_n),\\
\varepsilon'_{C_0}(u_l,x_l)  &= -\varepsilon'_{P}(v_0,y_1).
\end{align*}
Then by induction we have that $\varepsilon'_{T'_1}=\pm
\varepsilon_{T'_1}$. Thus $\varepsilon'_{P}=\pm\varepsilon_P$ and
$\varepsilon'_{C_0}=\pm\varepsilon_{C_0}$. Therefore
$\varepsilon'_T=\pm\varepsilon_T$; see Figure~\ref{Cycle-Tree}. This
shows that the direction $\varepsilon_T$ is unique up to opposite
sign.

``$\Leftarrow$": We also proceed by induction on the number of block
cycles. Let $\varepsilon_T$ be a direction on $T$, and let $W$ be a
closed walk on $T$ that uses each edge of block cycles once and each
edge of block paths twice. Then $(W,\varepsilon_T)$ is a directed
Eulerian walk. It is trivially true when $T$ has only one block
cycle, for the cycle has zero number of intersection vertices and is
balanced by Lemma~\ref{Directed-Walk-to-Sign}.

Assume that $T$ has two or more block cycles. Let
$(W_1,\varepsilon_{W_1})$ be the restriction of $(W,\varepsilon_W)$
to $T_1$ in \eqref{T1}. Then $(W_1,\varepsilon_{W_1})$ is a directed
walk, having either a sink or a source at $w_0$. Switch the sign of
the edge $z_1$ in $T$ and its orientation at $w_0$. We obtain a
directed Eulerian walk $(W'_1,\varepsilon'_{W_1})$ on $T'_1$ that
uses each edge of block cycles once and each edge of block paths
twice. By induction all block cycles of $T'_1$ satisfy the Parity
Condition. Thus all block cycles of $T_1$ other than $C_1$ satisfies
the Parity Condition. Let us switch the sign of $z_1$ back to the
sign of $z_1$ in $T$. Since $C_1$ has one fewer intersection vertex
in $T$ than that in $T'_1$, we see that $C_1$ satisfies the Parity
Condition in $T$. Since $(C_0,\varepsilon_T)$ has either a sink or a
source at $u_0$, it forces that $C_0$ is unbalanced. Hence $C_0$
also satisfies the Parity Condition.
\end{proof}

\begin{lemma}\label{Minimal-Directed-Eulerian-Walk-to-Midway-Cut-Point}
Let $W$ be a minimal Eulerian walk with a direction $\varepsilon_W$.
Then
\begin{enumerate}
\item[(a)] $(W,\varepsilon_W)$ is midway-back avoided.

\item[(b)] Each double vertex in $W$ is a cut-point of $\Sigma(W)$.
\end{enumerate}
\end{lemma}
\begin{proof}
(a) Suppose $(W,\varepsilon_W)$ is not midway-back avoided. Then $W$
can be written as
\[
W=u_0x_1u_1\ldots u_{\alpha-1} x_\alpha u_\alpha\ldots u_{\beta-1}
x_\beta u_\beta\ldots u_{\ell-1}x_\ell u_\ell,
\]
where $0\leq\alpha<\beta<\ell$, $u_\alpha=u_\beta$, and
$\varepsilon_W(u_\beta,x_\beta)=-\varepsilon_W(u_\alpha,x_{\alpha+1})$.
Then $(W',\varepsilon_W)$ is a directed closed walk contained
properly in $(W,\varepsilon_W)$ as multisets, where
\[
W'=u_\alpha x_{\alpha+1} u_{\alpha+1}\ldots u_{\beta-1} x_\beta
u_\beta.
\]
This is a contradiction.

(b) Suppose there is a double vertex in $W$ that is not a cut-point
of $\Sigma(W)$. Then we can write $W$ as
\[
W=u_\delta x_{\delta+1}u_{\delta+1}\dots x_{\mu}u_\mu \ldots
x_{\eta}u_\eta  \ldots x_{\nu}u_{\nu} \ldots
u_{\delta-1}x_{\delta}u_\delta,
\]
where $\delta<\mu<\eta<\nu$, $u_\delta=u_\eta$ and $u_\mu=u_\nu$;
see Figure~\ref{n-cut-point}. Since $(W,\varepsilon_W)$ is
midway-back avoided, we have
\begin{equation}\label{eta-delta-nu-mu}
\varepsilon_W(u_\eta, x_\eta)=\varepsilon_W(u_\delta, x_{\delta+1}),
\sp \varepsilon_W(u_\nu, x_\nu)=\varepsilon_W(u_\mu, x_{\mu+1}).
\end{equation}
Since $(W,\varepsilon_W)$ is directed, we must have
\begin{equation}\label{eta-delta-nu-mu-1}
\varepsilon_W(u_\delta, x_\delta)=\varepsilon_W(u_\eta, x_{\eta+1}),
\sp \varepsilon_W(u_\mu, x_\mu)=\varepsilon_W(u_\nu, x_{\nu+1}).
\end{equation}
Thus (\ref{eta-delta-nu-mu}) and (\ref{eta-delta-nu-mu-1}) imply
that $(W_1,\varepsilon_W)$ and $(W_2,\varepsilon_W)$ are directed
closed walks and are contained properly in $(W,\varepsilon_W)$ as
multisets, where
\[
W_1=u_\delta x_{\delta+1}u_{\delta+1}\ldots u_{\mu-1}x_\mu
u_{\mu}(u_\nu)x_{\nu}u_{\nu-1}\ldots u_{\eta+1}x_{\eta+1}u_\eta,
\]
\[
W_2=u_\eta x_{\eta}u_{\eta-1}\ldots u_{\mu+1}x_{\mu+1}
u_{\mu}(u_\nu)x_{\nu+1}u_{\nu+1}\ldots
u_{\delta-1}x_{\delta}u_\delta.
\]
This is a contradiction.
\end{proof}

\begin{thm}[Characterization of Minimal Eulerian
Walk]\label{Min-Directed-Eulerian-Walk-to-Eulerian-Cycle-Tree} Let
$W$ be a minimal Eulerian walk with a direction $\varepsilon_W$.
Then $\Sigma(W)$ is an Eulerian cycle-tree $T$, $W$ uses each edge
of block cycles once and each edge of block paths twice of $T$, and
$\varepsilon_W$ induces a direction $\varepsilon_T$ on $T$.
Moreover,
\begin{equation}\label{Walk-Flow-via-Eulerian-Cycle-Tree}
f_{(W,\,\varepsilon_W)} = [\varepsilon,\varepsilon_T]\cdot I_T.
\end{equation}
\end{thm}
\begin{proof}
It follows from
Lemma~\ref{Minimal-Directed-Eulerian-Walk-to-Midway-Cut-Point} and
Theorem~\ref{Midway-Cut-Point-to-Eulerian-Cycle-Tree} that
$\Sigma(W)$ is an Eulerian cycle-tree. Since
$\varepsilon_W(u,x)=\varepsilon_T(u,x)$ for vertex-edge pairs
$(u,x)$ on $T$, the identity
\eqref{Walk-Flow-via-Eulerian-Cycle-Tree} follows from definitions
\eqref{FW} and \eqref{Indicator-Function}.
\end{proof}

\begin{thm}[Minimality of Eulerian
Cycle-Tree]\label{Eulerian-Cycle-Tree-to-Minimality} Let $T$ be an
Eulerian cycle-tree with a direction $\varepsilon_T$. Then $T$ is
minimal in the sense that if $T_1$ is an Eulerian cycle-tree
contained in $T$ and block paths of $T_1$ are block paths of $T$
then $T_1=T$.

Moreover, if $W$ is a closed walk on $T$ that uses each edge of
block cycles once and each edge of block paths twice, then $W$ is a
minimal Eulerian walk. In particular, each closed walk found by FRA
is a minimal Eulerian walk.
\end{thm}
\begin{proof}
Suppose there is an Eulerian cycle-tree $T_1$ contained properly in
$T$, such that block paths of $T_1$ are block paths of $T$. Then
there exists an edge $e\in T\backslash T_1$, incident with a vertex
$v$ on a block cycle $C$ of $T_1$. The vertex $v$ must be an
intersection vertex in $T$ but not an intersection vertex in $T_1$.
Let $e_1,e_2$ be two edges (maybe be an identical loop) on $C$,
incident with $v$. Then $\varepsilon_T(v,e_1)=-\varepsilon_T(v,e_2)$
in $T_1$ and $\varepsilon_T(v,e_1)=\varepsilon_T(v,e_2)$ in $T$.
This is a contradiction.

Let $W$ be a required closed walk on $T$. It is clear that
$(W,\varepsilon_T)$ is a directed Eulerian walk by
Lemma~\ref{Cycle-Tree-to-Minimum-Closed-Walk} and by definition of
$\varepsilon_T$. Let $W_1$ be a minimal Eulerian walk on $T$,
contained properly in $W$ as multisets. Then $T_1=\Sigma(W_1)$ is
contained in $T$ and is an Eulerian cycle-tree by
Theorem~\ref{Min-Directed-Eulerian-Walk-to-Eulerian-Cycle-Tree}. It
is clear that block cycles of $T_1$ are block cycles of $T$. Note
that edges of block paths of $T$ are double edges in $W$, edges of
block paths of $T_1$ are double edges in $W_1$, and double edges in
$W_1$ must be double edges in $W$. It follows that block paths of
$T_1$ are block paths of $T$. Thus $T_1=T$ by the first part of the
theorem. Therefore $M(W_1)=M(W)$.
\end{proof}

\begin{cor}[Characterization and Classification of
Circuits]\label{Equiv-of-Elementary} Let $(W,\varepsilon_W)$ be a
minimal directed Eulerian walk. Then the following statements are
equivalent.
\begin{enumerate}
\item[(a)] $(W,\varepsilon_W)$ is
elementary.

\item[(b)] $f_{(W,\;\varepsilon_W)}$ is elementary.

\item[(c)] $\Sigma(W)$ is a circuit.
\end{enumerate}
Moreover, circuits are classified into Types I, II, III.
\end{cor}
\begin{proof}
(a) $\Leftrightarrow$ (b): Assume $(W,\varepsilon_W)$ is not
elementary, that is, there exists a minimal directed Eulerian walk
$(W_1,\varepsilon_{W_1})$ such that $\supp W_1\subsetneq \supp W$.
Since $\supp W_1=f_{(W_1,\,\varepsilon_{W_1})}$ and $\supp W=\supp
f_{(W,\,\varepsilon_{W})}$, then $\supp
f_{(W_1,\varepsilon_{W_1})}\subsetneq \supp
f_{(W,\varepsilon_{W})}$. This means that
$f_{(W,\,\varepsilon_{W})}$ is not elementary.

Conversely, assume $f_{(W,\;\varepsilon_W)}$ is not elementary, that
is, there is a flow $g$ such that $\supp g\subsetneq\supp
f_{(W,\,\varepsilon_W)}$. We may require $\Sigma(\supp g)$ to be
connected. By Lemma~\ref{Walk-to-Flow} there exists a directed
closed walk $(W_1,\varepsilon_g)$ on $\Sigma(\supp g)$ such that
$g=f_{(W_1,\,\varepsilon_g)}$. Since $\supp W_1=\supp g$ and $\supp
f_{(W,\,\varepsilon_W)}=\supp W$, then $\supp W_1\subsetneq\supp W$.
This means that $(W,\varepsilon_W)$ is not elementary.

(a) $\Leftrightarrow$ (c): If $(W,\varepsilon_W)$ is not elementary,
that is, there exists a minimal directed Eulerian walk
$(W_1,\varepsilon_{W_1})$ such that $\supp W_1\subsetneq\supp W$.
Then $\Sigma(W_1)$ is an Eulerian cycle-tree by
Theorem~\ref{Min-Directed-Eulerian-Walk-to-Eulerian-Cycle-Tree} and
is properly contained in $\Sigma(W)$. This means that $\Sigma(W)$ is
not a circuit.

Conversely, if $\Sigma(W)$ is not a circuit, that is, there exists
an Eulerian cycle-tree $T_1$ contained properly in $\Sigma(W)$. Let
$\varepsilon_{T_1}$ be a direction on $T_1$, and $W_1$ be a closed
walk that uses each edge of block cycles once and each edge of block
paths twice of $T_1$. Then $(W_1,\varepsilon_{T_1})$ is a minimal
directed Eulerian walk and $\supp W_1\subsetneq \supp W$. This means
that $(W,\varepsilon_W)$ is not elementary.

Now let $T$ be an Eulerian cycle-tree and be further a circuit. Then
$T$ contains at most one block path (of possible zero length).
Otherwise, suppose there are two or more block paths in $T$, then
one block path together with its two block cycles form an Eulerian
cycle-tree, which is properly contain in $T$; this means that $T$ is
not a circuit, a contradiction. If there is no block path in $T$,
then $T$ must be a single balanced cycle, which is a circuit of Type
I. If $T$ contains exactly one block path, the length of the block
path is either zero or positive.

In the case of zero length for the block path, $T$ consists of two
block cycles having a common vertex, which is a circuit of Type II.
In the case of positive length for the block path, $T$ consists of
two block cycles and the block path connecting them, which is a
circuit of Type III.
\end{proof}

\begin{thm}[Half-integer
Scale Decomposition]\label{Eulerin-Cycle-Tree-to-Half-Scale-Decom}
Let $T$ be an Eulerian cycle-tree with a direction $\varepsilon_T$.
Let $W$ be a closed walk on $T$ that uses each edge of block cycles
once and each edge of block paths twice. If $T$ is not a circuit,
then $W$ can be divided into the form
\begin{equation}
W=C_0P_1C_1P_2\cdots P_kC_{k}P_{k+1}, \sp k\geq 1,
\end{equation}
satisfying the following four conditions:
\begin{enumerate}
\item[(i)] $\{C_i\}$ is the collection of all end-block cycles of $T$ and $P_i$ are simple open paths of
positive lengths.

\item[(ii)] Each edge of non-end-block cycles appears in exactly one of the paths
$P_i$, and each edge of block paths appears in exactly two of the
paths $P_i$.

\item[(iii)] Each $(C_{i}P_{i+1}C_{i+1},\varepsilon_T)$ $(0\leq
i\leq k)$ is a directed circuit of Type III with $C_{k+1}=C_0$.

\item[(iv)] Half-integer scale decomposition
\begin{equation}
 I_T = \frac{1}{2} \sum_{i=0}^{k}
 I_{\Sigma(C_{i}P_{i+1}C_{i+1})}.
\end{equation}
\end{enumerate}
\end{thm}
\begin{proof}
We proceed by induction on the number of block paths in $T$,
including those of zero length. Note that $T$ is a circuit if there
is none or exactly one block path.
\begin{figure}[h]
\includegraphics[width=55mm]{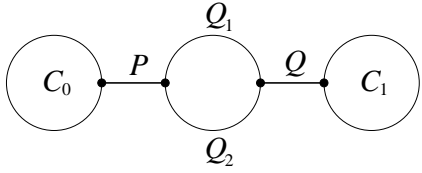}
\caption{An Eulerian cycle-tree with two block
paths.}\label{Two-Block-Paths}
\end{figure}
When $T$ has exactly two block paths, then $T$ has the form in
Figure~\ref{Two-Block-Paths}. Since $W$ crosses each cut-point from
one block to the other block, then $W$ can be written as
$W=C_0P_1C_1P_2$, where $P_1=PQ_1Q$, $P_2=Q^{-1}Q_2^{-1}P^{-1}$.
Then $C_0P_1C_1, C_1P_2C_0$ are circuits of Type III. We thus have
the decomposition
\[
I_T =\frac{1}{2}I_{\Sigma(C_0P_1C_1)} +\frac{1}{2}
I_{\Sigma(C_1P_2C_0)}.
\]

When $T$ has three or more block paths (of possible zero length),
choose an end-block cycle $C$ and a block path $P$ (of possible zero
length) having its initial vertex $u$ on $C$ and its terminal vertex
$v$ on another block cycle $C'$. Since $T$ has at least three block
paths, the cycle $C'$ cannot be a loop; so all edges of $C'$ are not
loops. Choose an edge $x$ on $C'$ at $v$, change the sign of $x$,
and remove the cycle $C$ and the internal part of $P$ from $T$. We
obtain an Eulerian cycle-tree $T'$; see Figures~\ref{Case-1} and
\ref{Case-2}. Then $W$ can be written as $W=CPW'P^{-1}$, where $W'$
is a closed walk on $T'$ that uses each edge of block cycles once
and each edge of block paths twice. Thus $(W',\varepsilon_{T'})$ is
a minimal Eulerian walk, where $\varepsilon_{T'}$ is a direction of
$T'$ and $\varepsilon_{T'}=\varepsilon_T$ except
$\varepsilon_{T'}(v,x) =-\varepsilon_{T}(v,x)$. By induction $W'$
can be written as
\[
W'=C'_0P'_1C'_1P'_2\cdots P'_kC'_{k}P'_{k+1},
\]
satisfying the conditions (i)--(iv). There are two cases: $C'$ is
either an end-block cycle of $T'$, or $C'$ is not an end-block cycle
of $T'$.

In the case that $C'$ is an end-block cycle of $T'$, we may assume
$C'_k=C'$, having its unique intersection vertex at $w$ in $T'$. Let
us write $C'$ as a closed path $C'_k=P'Q'$, where $P'$ is a path
from $v$ to $w$ on $C'$ and $Q'$ is the other path from $w$ to $v$
on $C'$. Note that $P'_k$ is a path whose terminal vertex is $w$,
and $P'_{k+1}$ is a path whose initial vertex is $w$; see
Figure~\ref{Case-1}.
\begin{figure}[h]
\includegraphics[width=55mm]{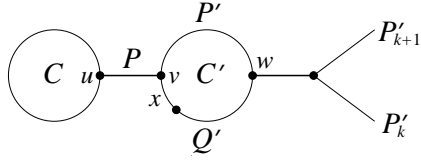}
\caption{$C'$ has two intersection vertices.}\label{Case-1}
\end{figure}
Set $C_{i}=C'_i$ $(0\leq i\leq k-1)$, $P_i=P'_i$ $(1\leq i\leq
k-1)$, and
\[
P_k=P'_kQ'P^{-1}, \hspace{2ex} C_k=C, \hspace{2ex}
P_{k+1}=PP'P'_{k+1}.
\]
Then $W=C_0P_1C_1P_2\cdots P_kC_{k}P_{k+1}$ is a closed walk on $T$,
satisfying the conditions (i)--(iv); see Figure~\ref{Case-1}.

In the case that $C'$ is not an end-block cycle of $T'$, we may
assume that $P'_{k+1}$ contains the vertex $v$ and the edge $x$. Let
us write $P'_{k+1}=P'Q'$, where $P'$ is a path whose terminal vertex
is $v$ and $Q'$ is a path whose initial vertex is $v$; see
Figure~\ref{Case-2}.
\begin{figure}[h]
\includegraphics[width=45mm]{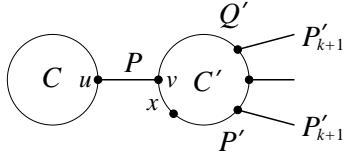}
\caption{$C'$ has more than two intersection
vertices.}\label{Case-2}
\end{figure}
Set $C_{i}=C'_i$ $(0\leq i\leq k)$, $P_i=P'_i$ $(1\leq i\leq k)$,
and
\[
P_{k+1}=P'P^{-1}, \hspace{2ex} C_{k+1}=C, \hspace{2ex} P_{k+2}=PQ'.
\]
Then $W=C_0P_1C_1P_2\cdots P_{k+1}C_{k+1}P_{k+2}$ is a closed walk
on $T$, satisfying the conditions (i)--(iv) with the direction
$\varepsilon_T$; see Figure~\ref{Case-2}.
\end{proof}

\noindent {\bf Problem.} An Eulerian cycle-tree is said to be {\em
bridgeless} if it does not contain block paths of positive length.
The indicator function of a bridgeless Eulerian cycle-tree has
constant value $1$ on its support. It should be interesting to
consider integral flows $f$ such that $\Sigma(f)$ is connected and
has no bridges; we may call such integral flows as {\em bridgeless
flows}. A bridgeless flow $f$ is said to be {\em bridgeless
decomposable} if there exist nontrivial bridgeless flows $f_1,f_2$
such that $f=f_1+f_2$, where $f_i$ have the same sign, that is,
$f_1\cdot f_2\geq 0$. It is particularly wanted to classify {\em
bridgeless indecomposable flows}, that is, the integral flows that
are not bridgeless decomposable.

\end{document}